\documentclass[11pt,twoside]{article}
\usepackage[utf8]{inputenc}
\usepackage[english]{babel}
\usepackage{fancyhdr}
\usepackage{amssymb,amsmath,amsthm,amsfonts}
\usepackage{mathtools}
\usepackage{hyperref}
\usepackage{caption}
\usepackage{enumerate}
\usepackage{algorithm}
\usepackage{algorithmicx}
\usepackage{algpseudocode}
\usepackage{xcolor}
\usepackage{enumitem}
\usepackage{framed}
\usepackage{xcolor}
\usepackage[caption=false]{subfig}
\usepackage[title]{appendix}
\usepackage[square,numbers]{natbib}

\hypersetup{
    colorlinks=false,
    urlcolor=black,
    pdfborder={0 0 0}
}

\newcommand{\email}[1]{E-mail: \href{mailto:#1}{\texttt{#1}}}

\def\R{\ensuremath{\mathbb R}}
\def\Rn{\ensuremath{\mathbb R^n}}

\def\S{\ensuremath{\mathcal S}}
\def\F{\ensuremath{\mathcal F}}
\def\sigmamax{\ensuremath{\sigma_{\text{max}}}}
\def\sigmamin{\ensuremath{\sigma_{\text{min}}}}
\def\lambdamin{\ensuremath{\lambda_{\text{min}}}}

\DeclarePairedDelimiter{\norm}{\lVert}{\rVert}

\DeclareMathOperator*{\argmin}{Argmin}

\newtheorem{definition}{Definition}
\newtheorem{assumption}{Assumption}
\newtheorem{condition}{Condition}
\newtheorem{theorem}{Theorem}
\newtheorem{proposition}{Proposition}
\newtheorem{lemma}{Lemma}
\newtheorem{remark}{Remark}

\theoremstyle{definition}

\begin{document}
\thispagestyle{plain}

\setcounter{page}{1}

{\centering
%-----------------------%
% INSERT HERE THE TITLE %
%-----------------------%
{\LARGE \bfseries Full Convergence of Regularized Methods for Unconstrained Optimization}

\bigskip\bigskip
%-------------------------%
% INSERT HERE THE AUTHORS %
%-------------------------%
Andrea Cristofari$^*$
\bigskip

}

%----------------------------------%
% INSERT HERE AUTHORS' INFORMATION %
%----------------------------------%
\begin{center}
\small{\noindent$^*$Department of Civil Engineering and Computer Science Engineering \\
University of Rome ``Tor Vergata'' \\
Via del Politecnico, 1, 00133 Rome (Italy) \\
\email{andrea.cristofari@uniroma2.it} \\
}
\end{center}

\bigskip\par\bigskip\par
\noindent \textbf{Abstract.}
Typically, the sequence of points generated by an optimization algorithm may have multiple limit points. 
Under convexity assumptions, however, (sub)gradient methods are known to generate a convergent sequence of points.
In this paper, we extend the latter property to a broader class of algorithms.
Specifically, we study unconstrained optimization methods that use local quadratic models regularized by a power $r \ge 3$ of the norm of the step.
In particular, we focus on the case where only the objective function and its gradient are evaluated.
Our analysis shows that, by a careful choice of the regularized model at every iteration, the whole sequence of points generated by this class of algorithms converges if the objective function is pseudoconvex.
The result is achieved by employing appropriate matrices to ensure that the sequence of points is variable metric quasi-Fej{\'e}r monotone.

 Variable metric quasi-Fej{\'e}r monotonicity.

\section{Introduction}
Let us consider the following unconstrained optimization problem:
\begin{equation}\label{prob}
\min_{x \in \Rn} f(x),
\end{equation}
where the objective function $f \colon \Rn \to \R$ has a continuous gradient $\nabla f(x)$.
Among the huge number of algorithms proposed in the literature to solve problem~\eqref{prob},
a popular approach is represented by the class of \textit{regularized methods}, which are characterized by the use of suitable models combined with a regularization term.
In particular, here we consider \textit{quadratic models} regularized by a power $r$ of the norm of the step.
Namely, to update a point $x_k$ generated at iteration $k$, we use the following regularized model:
\begin{equation}\label{mk}
m_k(s) := f(x_k) + \nabla f(x_k)^T s + \frac 12 s^T Q_k s + \frac{\sigma_k}r \norm s^r,
\end{equation}
with symmetric matrix $Q_k \in \R^{n \times n}$ and regularization parameter $\sigma_k \ge 0$,
while, here and in the rest of the paper, $\norm{\cdot}$ denotes the Euclidean norm.
Then, a step $s_k$ is computed as an approximate minimizer of $m_k$ in order to set the next iterate as
$x_{k+1} = x_k + s_k$ if an appropriate objective decrease is achieved.

In this paper, we focus on the case where \textit{only the objective function and its gradient are evaluated} at each iteration.
In such a setting, convergence and complexity results of regularized methods can be found
in~\cite{cartis:2011a,cartis:2011b,cartis:2012b,cartis:2017}.

Our analysis shows that, when $f$ is pseudoconvex and $r \ge 3$,
\textit{the whole sequence $\{x_k\}$ generated by the regularized methods converges} if an appropriate sequence of matrices $\{Q_k\}$ is employed.
Specifically, $\{Q_k\}$ must be chosen so as to guarantee that $\{x_k\}$ is a \textit{variable metric quasi-Fej{\'e}r monotone} sequence.
%The proposed condition on $\{Q_k\}$ can be easily satisfied in practice by using, for example, a constant positive definite matrix for all iterations.

In the literature, it is known that the whole sequence of points converges, under convexity assumptions,
for (sub)gradient methods~\cite{alber:1998,bertsekas:2015,burachik:1995,iusem:2003,nesterov:2013,polyak:1987}.
Our result hence extends that property to a broader class of regularized methods.
Interestingly, our analysis also shows that the convergence of the algorithm does not require the assumption of bounded level sets, which is typically
needed to obtain the $\mathcal O(1/k)$ convergence rate of regularized methods using first-order information in the convex case~\cite{cartis:2012b}.

Let us also remark that the matrix $Q_k$ used in model~\eqref{mk} is not required to approximate the Hessian (which may not exist in our setting, since we only assume access to $f$ and $\nabla f$). 
This differentiates our approach from methods that explicitly employ Hessian approximations, e.g., via finite differences~\cite{grapiglia:2022}.

The rest of the paper is organized as follows.
In Section~\ref{sec:fejer}, we review the notion of Fej{\'e}r monotonicity and its extensions.
In Section~\ref{sec:reg_methods}, we describe the class of regularized methods under analysis.
In Section~\ref{sec:analysis}, we give the main result of the paper showing the convergence of the whole sequence of points generated by the considered class of algorithms.
In Section~\ref{sec:alg_example}, we provide an example of a regularized method that generates a sequence of points converging to a minimizer of the objective function.
In Section~\ref{sec:matrix}, we describe an iterative procedure to compute appropriate matrices that satisfy the conditions required for the considered class of algorithms.
In Section~\ref{sec:num}, we analyze the practical performance of the proposed algorithmic scheme under different choices for the sequence of matrices.
Finally, some conclusions are drawn in Section~\ref{sec:conclusions}.

\section{Fej{\'e}r Monotonicity and Extensions}\label{sec:fejer}
Let us consider a sequence $\{x_k\} \subseteq \Rn$ generated by an optimization algorithm.
A common tool to prove the convergence of $\{x_k\}$ comes from the notion of \textit{Fej{\'e}r monotonicity}, defined as follows.

\begin{definition}[Fej{\'e}r monotonicity]\label{def:fejer}
Given a non-empty set $\F \subseteq \Rn$, we say that $\{x_k\} \subseteq \Rn$ is a \textit{Fej{\'e}r monotone sequence} with respect to $\F$ if,
for all $y \in \F$, it holds that
\[
\norm{x_{k+1} - y}^2 \le \norm{x_k-y}^2 \quad \forall k \ge 0.
\]
\end{definition}

It is well known that, if $\{x_k\}$ is a Fej{\'e}r monotone sequence with respect to $\F$ and $\{x_k\}$ has a limit point $x^* \in \F$,
then the whole sequence $\{x_k\}$ converges to $x^*$ (see, e.g., \cite{bertsekas:2015}).

In the literature, Fej{\'e}r monotonicity has been used to show the convergence of some optimization algorithms under appropriate assumptions.
In particular, when the objective function $f$ is convex with Lipschitz continuous gradient,
the sequence $\{x_k\}$ generated by the gradient method is Fej{\'e}r monotone
if the stepsize is chosen within an appropriate interval, hence converging to a minimizer of $f$~\cite{nesterov:2013}.
For a non-smooth convex objective function $f$, also the subgradient method produces a Fej{\'e}r monotone sequence $\{x_k\}$ when using the Polyak stepsize,
thus converging to a minimizer of $f$~\cite{polyak:1987}.

Some extensions of Fej{\'e}r monotonicity have been proposed in the literature in order to guarantee the same convergence properties for a broader family of sequences.
A first extension is represented by \textit{quasi-Fej{\'e}r monotonicity}, which is defined below.

\begin{definition}[quasi-Fej{\'e}r monotonicity]\label{def:quasi_fejer}
Given a non-empty set $\F \subseteq \Rn$, we say that $\{x_k\} \subseteq \Rn$ is a \textit{quasi-Fej{\'e}r monotone sequence} with respect to $\F$ if,
for all $y \in \F$, it holds that
\[
\norm{x_{k+1} - y}^2 \le \norm{x_k-y}^2 + \epsilon_k \quad \forall k \ge 0,
\]
where $\{\epsilon_k\}$ is a sequence of scalars such that
\[
\{\epsilon_k\} \subseteq [0,\infty), \quad \sum_{k=0}^{\infty} \epsilon_k < \infty.
\]
\end{definition}

Using quasi-Fej{\'e}r monotonicity, it is still possible to show that, under convexity assumptions,
the whole sequence generated by gradient methods converges to a minimizer of the objective function $f$
using suitable line search techniques (which include the classical Armijo line search)~\cite{burachik:1995,iusem:2003}.
In the non-smooth case, also subgradient methods with a diminishing stepsize generate a quasi-Fej{\'e}r monotone sequence,
under convexity assumptions, hence converging to a minimizer of the objective function~\cite{alber:1998,bertsekas:2015}.
Further, quasi-Fej{\'e}r monotonicity has been used even in vector optimization to show the convergence of the whole sequence of points
generated by the steepest descent method under convexity assumptions~\cite{drummond:2005}.

Additionally, a more general definition of quasi-Fej{\'e}r monotonicity was given in~\cite{iusem:1994} using
a divergence measure in place of the quadratic norm appearing in Definition~\ref{def:quasi_fejer},
whereas a comparison between different types of quasi-Fej{\'e}r monotonicity
and an analysis of their use in optimization can be found in~\cite{combettes:2001} .

Going beyond Definition~\ref{def:quasi_fejer}, a further extension
has been proposed in~\cite{combettes:2013} with the notion of \textit{variable metric quasi-Fej{\'e}r monotonicity} which
employs variable metrics to measure the distance between points.
Using that tool, weak and strong convergence results for sequences in a real Hilbert space were established in~\cite{combettes:2013}.

In the following, we report a simplified definition of variable metric quasi-Fej{\'e}r monotonicity adapted from~\cite{combettes:2013} to our purposes,
using, here and in the rest of the paper, the notation $\norm Q$ to denote the norm induced by the vector Euclidean norm for any matrix $Q$,
while $\norm{\cdot}_Q$ indicates the norm induced by a symmetric positive definite matrix $Q \in \R^{n \times n}$, that is,
$\norm v_Q = (v^T Q v)^{1/2}$ for all $v \in \Rn$.

\begin{definition}[variable metric quasi-Fej{\'e}r monotonicity]\label{def:var_quasi_fejer}
Given a non-empty set $\F \subseteq \Rn$ and a sequence of symmetric matrices $\{Q_k\} \subseteq \R^{n \times n}$ such that
\begin{align}
Q_k \succeq aI \quad \forall k \ge 0, \quad & \text{with} \quad a > 0, \label{Q_var_quasi_fejer} \\
\norm{Q_k} \le b \quad \forall k \ge 0, \quad & \text{with} \quad b < \infty, \label{Q_bounded}
\end{align}
we say that $\{x_k\} \subseteq \Rn$ is a \textit{variable metric quasi-Fej{\'e}r monotone sequence} with respect to $\F$
relative to $\{Q_k\}$ if, for all $y \in \F$, it holds that
\begin{equation}\label{var_quasi_fejer_ineq}
\norm{x_{k+1} - y}_{Q_{k+1}}^2 \le (1+\psi_k) \norm{x_k-y}^2_{Q_k} + \epsilon_k \quad \forall k \ge 0,
\end{equation}
where $\{\psi_k\}$ and $\{\epsilon_k\}$ are two sequences of scalars such that
\begin{gather}
\{\psi_k\} \subseteq [0,\infty), \quad \sum_{k=0}^{\infty} \psi_k < \infty, \label{eta_summable} \\
\{\epsilon_k\} \subseteq [0,\infty), \quad \sum_{k=0}^{\infty} \epsilon_k < \infty. \label{eps_summable}
\end{gather}
\end{definition}

Variable metric quasi-Fej{\'e}r monotonicity will be central to our analysis.
Specifically, we will use the following result to ensure the convergence of a sequence of points.

\begin{theorem}\label{th:variable_quasi_fejer}
Let $\{x_k\}$ be a variable metric quasi-Fej{\'e}r monotone sequence with respect to a set $\F$ relative to a sequence of matrices $\{Q_k\}$,
according to Definition~\ref{def:var_quasi_fejer}.
The following holds:
\begin{enumerate}[label=(\roman*)]
\item $\displaystyle{\norm{x_k-y} \le R}$ for all $k \ge 0$ and all $y \in \F$, where \\
    $\displaystyle{R := \Biggl(\frac 1a \prod_{j=0}^{\infty} (1+\psi_j) \biggl(\norm{x_0-y}_{Q_0}^2 + \sum_{i=0}^{\infty}\epsilon_i\biggr)\Biggr)^{1/2} < \infty}$
	with $a$, $\{\psi_k\}$ and $\{\epsilon_k\}$ given in Definition~\ref{def:var_quasi_fejer}; \label{prop_variable_quasi_fejer_R0}
\item $\{x_k\}$ is bounded; \label{prop_variable_quasi_fejer_bounded}
\item if $x^* \in \F$ is a limit point of $\{x_k\}$, then $\{x_k\}$ converges to $x^*$. \label{prop_variable_quasi_fejer_conv}
\end{enumerate}
\end{theorem}

\begin{proof}{Proof}
Fix any $y \in \F$. Let $a$, $b$, $\{\psi_k\}$ and $\{\epsilon_k\}$ be as given in Definition~\ref{def:var_quasi_fejer}.
Since $\{\psi_k\}$ satisfies~\eqref{eta_summable}, we can define
\begin{equation}\label{prod_conv}
1 \le \zeta := \prod_{i=0}^{\infty} (1+\psi_i) < \infty.
\end{equation}
Applying the variable metric quasi-Fej{\'e}r inequality~\eqref{var_quasi_fejer_ineq} recursively from an index $k$ to an index $t \le k$, we have that
\[
\norm{x_k-y}_{Q_k}^2 \le \prod_{i=t}^{k-1} (1+\psi_i) \norm{x_t-y}_{Q_t}^2 + \sum_{i=t}^{k-2}\prod_{j=i+1}^{k-1}(1+\psi_j)\epsilon_i + \epsilon_{k-1} \quad \forall k \ge t \ge 0.
\]
Using~\eqref{prod_conv} together with the fact that $\{\psi_k\}$ and $\{\epsilon_k\}$ are sequences of non-negative real numbers from~\eqref{eta_summable}
and~\eqref{eps_summable}, respectively, it follows that
\[
\norm{x_k-y}_{Q_k}^2 \le \zeta \norm{x_t-y}_{Q_t}^2 + \zeta \sum_{i=t}^{k-2}\epsilon_i + \epsilon_{k-1}
\le \zeta \biggl(\norm{x_t-y}_{Q_t}^2 + \sum_{i=t}^{\infty}\epsilon_i\biggr) \quad \forall k \ge t \ge 0.
\]
Since, from~\eqref{Q_var_quasi_fejer}, we have $\norm{x_k-y}_{Q_k}^2 \ge a \norm{x_k-y}^2$ for all $k \ge 0$, we get
\begin{equation}\label{recurs_variable_quasi_fejer_0}
\norm{x_k-y}^2 \le \frac{\zeta}a \biggl(\norm{x_t-y}_{Q_t}^2 + \sum_{i=t}^{\infty}\epsilon_i\biggr) < \infty \quad \forall k \ge t \ge 0,
\end{equation}
where the last inequality follows from the fact that $\{\epsilon_k\}$ satisfies~\eqref{eps_summable}
together with the fact that $\zeta/a$ is finite from~\eqref{prod_conv} and~\eqref{Q_var_quasi_fejer}.
Hence, items~\ref{prop_variable_quasi_fejer_R0}--\ref{prop_variable_quasi_fejer_bounded} hold.

Now, let $x^* \in \F$ be a limit point of $\{x_k\}$, that is, there exists a subsequence $\{x_k\}_K \to x^*$, with $K \subseteq \{0,1,\ldots\}$.
Using~\eqref{eps_summable}, for any $\omega > 0$ there exists an index $k_{\omega}$ such that
\begin{align*}
\norm{x_{k_{\omega}}-x^*}^2 & \le \frac{a \omega}{2b\zeta}, \\
\sum_{i=k_{\omega}}^{\infty}\epsilon_i & \le \frac{a \omega}{2\zeta}.
\end{align*}
Hence, using~\eqref{recurs_variable_quasi_fejer_0} with $y = x^*$ and $t=k_{\omega}$, for all $k \ge k_{\omega}$ we have that
\[
\begin{aligned}
\norm{x_k-x^*}^2 \le \frac{\zeta}a\biggl(\norm{x_{k_{\omega}}-x^*}_{Q_{k_{\omega}}}^2 + \sum_{i=k_{\omega}}^{\infty}\epsilon_i\biggr)
& \le \frac{\zeta}a\biggl(\norm{Q_{k_{\omega}}} \norm{x_{k_{\omega}}-x^*}^2 + \sum_{i=k_{\omega}}^{\infty}\epsilon_i\biggr) \\
& \le \frac{\zeta}a\biggl(b \norm{x_{k_{\omega}}-x^*}^2 + \sum_{i=k_{\omega}}^{\infty}\epsilon_i\biggr) \\
& \le \frac{\zeta}a\biggl(\frac{a \omega}{2\zeta} + \frac{a \omega}{2\zeta}\biggr) \\
& = \omega,
\end{aligned}
\]
where we have used~\eqref{Q_bounded} in the third inequality.
Since this is true for any arbitrary $\omega > 0$, we conclude that $\{x_k\}$ converges to $x^*$, thus proving item~\ref{prop_variable_quasi_fejer_conv}.
\end{proof}

In the next lemma, we give necessary and sufficient conditions for $\{x_k\}$ to be a variable metric quasi-Fej{\'e}r monotone sequence.

\begin{lemma}\label{lemma:eq_var_quasi_fejer}
Let $\F \subseteq \Rn$ be a non-empty set and let $\{Q_k\} \subseteq \R^{n \times n}$ be a sequence of symmetric matrices satisfying~\eqref{Q_var_quasi_fejer}--\eqref{Q_bounded}.
Then, $\{x_k\}$ is a variable metric quasi-Fej{\'e}r monotone sequence with respect to $\F$ relative to $\{Q_k\}$, according to Definition~\ref{def:var_quasi_fejer},
if and only if, for all $y \in \F$, it holds that
\begin{equation}\label{eq_var_quasi_fejer_ineq}
\norm{x_{k+1} - y}_{Q_{k+1}}^2 \le (1+\psi_k) \norm{x_k-y}^2_{Q_k} + \theta_k \norm{x_k-y} + \epsilon_k \quad \forall k \ge 0,
\end{equation}
where $\{\psi_k\}$ and $\{\epsilon_k\}$ satisfy~\eqref{eta_summable} and~\eqref{eps_summable}, respectively,
while $\{\theta_k\}$ is a sequence of scalars such that
\begin{equation}\label{theta_summable}
\{\theta_k\} \subseteq [0,\infty), \quad \sum_{k=0}^{\infty} \theta_k < \infty.
\end{equation}
\end{lemma}

\begin{proof}{Proof}
Take any $y \in \F$ and let us distinguish the necessary and sufficient implications separately.
\begin{itemize}
\item Necessity. If Definition~\ref{def:var_quasi_fejer} is satisfied, then~\eqref{eq_var_quasi_fejer_ineq}--\eqref{theta_summable}
    hold by using $\theta_k = 0$ for all $k \ge 0$.
\item Sufficiency. Assume that~\eqref{eq_var_quasi_fejer_ineq}--\eqref{theta_summable} hold for a given $y \in \F$.
    We can partition the index set $\{0,1,\ldots\}$ into two subsets $K_1$ and $K_2$ such that
    \begin{gather*}
    k \in K_1 \quad \Leftrightarrow \quad \norm{x_k-y} < 1, \\
    k \in K_2 \quad \Leftrightarrow \quad \norm{x_k-y} \ge 1.
    \end{gather*}
    Observe that, in view of~\eqref{Q_var_quasi_fejer}, we can write
    \[
    \norm{x_k-y}_{Q_k}^2 \ge a \norm{x_k-y}^2 \ge a \norm{x_k-y} \quad \forall k \in K_2.
    \]
    Then, using~\eqref{eq_var_quasi_fejer_ineq}, we get
    \begin{align*}
    \norm{x_{k+1} - y}_{Q_{k+1}}^2 \le (1+\psi_k) \norm{x_k-y}^2_{Q_k} + \theta_k + \epsilon_k \quad & \forall k \in K_1, \\
    \norm{x_{k+1} - y}_{Q_{k+1}}^2 \le \biggl(1+\psi_k+\frac{\theta_k}{a}\biggr) \norm{x_k-y}^2_{Q_k} + \epsilon_k \quad & \forall k \in K_2.
    \end{align*}
    Since $a>0$ and $\theta_k \ge 0$ for all $k \ge 0$, it follows that
    \begin{align*}
    \norm{x_{k+1} - y}_{Q_{k+1}}^2 \le \biggl(1+\psi_k+\frac{\theta_k}{a}\biggr) \norm{x_k-y}^2_{Q_k} + \theta_k + \epsilon_k \quad & \forall k \in K_1, \\
    \norm{x_{k+1} - y}_{Q_{k+1}}^2 \le \biggl(1+\psi_k+\frac{\theta_k}{a}\biggr) \norm{x_k-y}^2_{Q_k} + \theta_k + \epsilon_k \quad & \forall k \in K_2.
    \end{align*}
    Namely, recalling that $K_1 \cup K_2 = \{0,1,\ldots\}$, we have that
    \[
    \norm{x_{k+1} - y}_{Q_{k+1}}^2 \le (1+\tilde\psi_k) \norm{x_k-y}^2_{Q_k} + \tilde \epsilon_k \quad \forall k \ge 0,
    \]
    where
    \[
    \tilde \psi_k = \psi_k+\frac{\theta_k}a \quad \text{and} \quad \tilde \epsilon_k = \epsilon_k + \theta_k.
    \]
    Using~\eqref{eta_summable}, \eqref{eps_summable} and~\eqref{theta_summable}, it follows that $\tilde \psi_k$ and $\tilde \epsilon_k$ are two sequences of scalars such that
    \begin{gather*}
    \tilde \psi_k \subseteq [0,\infty), \quad \sum_{k=0}^{\infty}\tilde \psi_k < \infty, \\
    \tilde \epsilon_k \subseteq [0,\infty), \quad \sum_{k=0}^{\infty}\tilde \epsilon_k < \infty.
    \end{gather*}
    Hence, $\{x_k\}$ is a variable metric quasi-Fej{\'e}r monotone sequence with respect to $\F$ relative to $\{Q_k\}$, according to Definition~\ref{def:var_quasi_fejer}
\end{itemize}
\end{proof}

\section{Regularized Methods}\label{sec:reg_methods}
In this section, we describe a simple framework for regularized methods which will be the subject of our analysis.
The scheme we consider here is intentionally given \textit{as general as possible} to include only the ingredients needed to ensure the convergence of
the generated sequence $\{x_k\}$ to a point $x^* \in \Rn$.
Other properties, such as the conditions under which $x^*$ is a minimizer of $f$ and the convergence rate,
will be analyzed in the next section.

Given a regularization power $r \ge 3$, at every iteration $k$ we consider the model $m_k$ defined as in~\eqref{mk} by choosing
appropriate regularization parameter $\sigma_k \ge 0$ and symmetric matrix $Q_k \in \R^{n \times n}$.
For the sake of convenience, let us also define the quadratic part in the regularized model as
\begin{equation}\label{qk}
q_k(s) := f(x_k) + \nabla f(x_k)^T s + \frac 12 s^T Q_k s,
\end{equation}
so that, from~\eqref{mk} and~\eqref{qk}, we can write
\[
m_k(s) = q_k(s) + \frac{\sigma_k}r \norm s^r.
\]
In order to get the next point $x_{k+1}$, we compute $s_k$ as an inexact minimizer of $m_k(s)$.
In particular, here we require $s_k$ to satisfy an approximate first-order condition, that is,
\begin{equation}\label{sk_cond_gmk}
\norm{\nabla m_k(s_k)} \le \tau \norm{s_k} \min\{\norm{s_k},1\},
\end{equation}
with $\tau \ge 0$.
Note that, in the considered setting where $r \ge 3$, the function $m_k(s)$ is coercive and has a global minimizer.
Consequently, a point $s_k$ satisfying~\eqref{sk_cond_gmk} can be obtained in finite time by applying a descent method for unconstrained optimization.

Then, the next point $x_{k+1}$ is obtained by means of an update rule ensuring that
\begin{equation}\label{f_decr}
x_{k+1} \ne x_k \quad \Rightarrow \quad f(x_k) - f(x_{k+1}) \ge \eta (m_k(0)-m_k(s_k)) > 0,
\end{equation}
with $\eta > 0$. We see that, by the update rule~\eqref{f_decr},
we can set $x_{k+1} = x_k + s_k$ only if the objective decrease given in~\eqref{f_decr} holds,
otherwise we have to set $x_{k+1} = x_k$.
Then, we say that $k$ is a \textit{successful iteration} if $x_{k+1} = x_k + s_k$,
denoting the set of successful iterations by $\S$. Namely,
\begin{equation}\label{S_def}
k \in \S \quad \Leftrightarrow \quad x_{k+1} = x_k + s_k.
\end{equation}
The scheme is reported in Algorithm~\ref{alg:reg_method}.

\begin{algorithm}
\caption{General scheme of a regularized method}\label{alg:reg_method}
\begin{algorithmic}[1]
\State given $r \in [3,\infty)$, $\tau \in [0,\infty)$, $\eta > 0$ and $\sigmamax \in [0,\infty)$
\State choose $x_0 \in \Rn$
\For{$k = 0,1,\ldots$}
\State compute $\sigma_k \in [0,\sigmamax]$
\State compute a symmetric matrix $Q_k \in \R^{n \times n}$
\State compute $s_k \in \Rn$ satisfying~\eqref{sk_cond_gmk}
\State set either $x_{k+1} = x_k + s_k$ or $x_{k+1} = x_k$ in order to satisfy~\eqref{f_decr}
\EndFor
\end{algorithmic}
\end{algorithm}

As mentioned at the beginning of this section, Algorithm~\ref{alg:reg_method} is a general framework in which
a number of details are intentionally left unspecified.
Namely, besides the choice of $Q_k$---which will be discussed in the next subsection---two main issues should be addressed:
how to compute $\sigma_k$ and how to guarantee the objective decrease condition~\eqref{f_decr}.

Regarding the choice of $\sigma_k$, various strategies can be found in the literature.
Common approaches make use of adaptive rules which dynamically update $\sigma_k$ during the iterations in a trust-region fashion (see, e.g.,~\cite{cartis:2011a}).
Moreover, in Section~\ref{sec:alg_example}, we will describe an algorithm where $\sigma_k$ can assume any value in $[0,\sigmamax]$
with an appropriate choice of the matrix $Q_k$ under Lipschitz continuity of $\nabla f$.

To guarantee the objective decrease condition~\eqref{f_decr}, usually $f(x_k+s_k) - f(x_k)$ is compared to the decrease achieved by
either the regularized model $m_k$ or the quadratic model $q_k$, in order to decide whether to set $x_{k+1} = x_k + s_k$ or $x_{k+1} = x_k$.
In more detail, a first possibility~\cite{cartis:2011a,cartis:2017} is to compute
\begin{equation}\label{rho1}
\rho_k = \frac{f(x_k)-f(x_k+s_k)}{m_k(0) - m_k(s_k)}
\end{equation}
and set $x_{k+1} = x_k + s_k$ if $\rho_k \ge \eta$.
By requiring $s_k$ to be such that $m_k(0) - m_k(s_k) > 0$, it is straightforward to see that this update rule satisfies~\eqref{f_decr}.
A second possibility~\cite{birgin:2017,cartis:2019} is to use
\begin{equation}\label{rho2}
\rho_k = \frac{f(x_k)-f(x_k+s_k)}{q_k(0) - q_k(s_k)}
\end{equation}
and set $x_{k+1} = x_k + s_k$ if $\rho_k \ge \eta$.
By requiring $s_k$ to be such that $q_k(0) - q_k(s_k) > 0$, also this update rule satisfies~\eqref{f_decr} since
\[
q_k(0) - q_k(s_k) = m_k(0) - m_k(s_k) + \frac{\sigma_k}r \norm{s_k}^r \ge m_k(0) - m_k(s_k).
\]
Hence, by means of the general condition~\eqref{f_decr} in Algorithm~\ref{alg:reg_method}, all the above update rules are included in our analysis.

Also note that Algorithm~\ref{alg:reg_method} encompasses the classical gradient method
defined by the $k$th iteration $x_{k+1} = x_k - \alpha_k \nabla f(x_k)$, with the stepsize $\alpha_k > 0$ chosen by an Armijo-type line search such that
$f(x_{k+1}) \le f(x_k) - \gamma \alpha_k \|\nabla f(x_k)\|^2$, where $\gamma \in (0,1)$ (see, e.g., \cite{bertsekas:1999}).
Specifically, we can use $\tau = 0$ while setting $\sigma_k = 0$ and $Q_k = \alpha_k^{-1} I$ for every iteration $k$.
It is straightforward to verify that, for every iteration $k$, such a parameter choice leads to $s_k = -\alpha_k \nabla f(x_k)$
and $m_k(0)-m_k(s_k) = (\alpha_k/2) \|\nabla f(x_k)\|^2$, implying that~\eqref{f_decr} is satisfied with $\eta = 2\gamma$.
If $\nabla f$ is Lipschitz continuous with constant $L$,
then we can also use, for every iteration $k$, a constant stepsize $\alpha_k = \alpha \in (0,2(1-\gamma)/L]$
as it still guarantees the objective decrease required by the Armijo-type line search~\cite{bertsekas:1999}.

\section{Analysis of Full Convergence}\label{sec:analysis}
In this section, we will show that the whole sequence of points $\{x_k\}$ generated by Algorithm~\ref{alg:reg_method} converges under appropriate hypotheses.

The key element in our analysis is an adequate choice of $\{Q_k\}$, reported in Condition~\ref{cond:Q} below, to
guarantee that $\{x_k\}$ is a variable metric quasi-Fej{\'e}r monotone sequence, according to Definition~\ref{def:var_quasi_fejer}, provided that $f$ is pseudoconvex.
This, in view of Theorem~\ref{th:variable_quasi_fejer}, will allow us to state the convergence of $\{x_k\}$.

\begin{condition}\label{cond:Q}
The sequence of matrices $\{Q_k\}$ used in Algorithm~\ref{alg:reg_method} is such that, for every iteration $k \ge 0$, it holds that
\begin{subequations}
\begin{align}
Q_k & \succeq aI, \label{Q_pos_def_assump} \\
Q_{k+1} & \preceq (1+\psi_k) Q_k, \label{Q_prec_assump}
\end{align}
\end{subequations}
where $\{\psi_k\}$ satisfies~\eqref{eta_summable} and
\begin{equation}\label{a_tau}
a > 2\tau.
\end{equation}
\end{condition}

From Condition~\ref{cond:Q}, we see that~\eqref{Q_pos_def_assump} and~\eqref{a_tau} impose that all the eigenvalues of  $Q_k$ are greater than $2 \tau$,
where $\tau$ is the parameter used in~\eqref{sk_cond_gmk} to control the degree of approximation when
computing $s_k$ as an inexact minimizer of $m_k$.
Moreover, \eqref{Q_prec_assump} relates two consecutive matrices and guarantees that
\begin{equation}\label{norm_Qk_succ}
\norm v^2_{Q_{k+1}} \le (1+\psi_k) \norm v^2_{Q_k} \quad \forall v \in \Rn.
\end{equation}

\begin{remark}\label{rem:Q}
The easiest way to satisfy Condition~\ref{cond:Q} is choosing a matrix $Q \succ 2\tau I$ and setting $Q_k = Q$ for all $k \ge 0$.
A more elaborate iterative procedure to compute $\{Q_k\}$ will be presented in Section~\ref{sec:matrix}.
\end{remark}

\subsection{Preliminary Results}
Let us provide some preliminary results which will be useful to get the main result in the next subsection.
First note that, for every iteration $k$, we can express the gradient of the regularized model~\eqref{mk} as follows:
\begin{equation}\label{grad_mk}
\nabla m_k(s) = \nabla f(x_k) + Q_k s + \sigma_k \norm s^{r-2} s.
\end{equation}

To begin with, in the following lemma we show that the sequence of matrices $\{Q_k\}$ is bounded.

\begin{lemma}\label{lemma:Q_ub}
Let $\{Q_k\}$ be the sequence of matrices used in Algorithm~\ref{alg:reg_method}.
If Condition~\ref{cond:Q} is satisfied, then $b \in \R$ exists such that
\[
\|Q_k\| \le b \quad \forall k \ge 0.
\]
\end{lemma}

\begin{proof}{Proof}
Fix any iteration $k \ge 1$. Since $Q_k \in \R^{n \times n}$ is symmetric and positive definite from Condition~\ref{cond:Q}, we can define $v_k$
as a unit-norm eigenvector of $Q_k$ associated to the largest eigenvalue of $Q_k$.
Namely,
\begin{equation}\label{norm_Qk}
v_k^T Q_k v_k = \norm{Q_k}, \quad \norm{v_k} = 1.
\end{equation}
Using~\eqref{norm_Qk} and~\eqref{norm_Qk_succ}, we can write
\begin{equation}\label{vQv}
\norm{Q_k} = v_k^T Q_k v_k = \norm{v_k}^2_{Q_k} \le (1+\psi_{k-1}) \norm{v_k}^2_{Q_{k-1}} = (1+\psi_{k-1}) v_k^T Q_{k-1} v_k.
\end{equation}
Since $\{\psi_k\}$ satisfies~\eqref{eta_summable} in view of Condition~\ref{cond:Q}, we can define
\[
1 \le \zeta := \prod_{i=0}^{\infty} (1+\psi_i) < \infty
\]
and, applying~\eqref{vQv} recursively, we get
\[
\norm{Q_k} \le \prod_{i=0}^{k-1} (1+\psi_i) v_k^T Q_0 v_k \le \zeta \norm{Q_0},
\]
where, in the last inequality, we have used the fact that $\norm{v_k} = 1$ from~\eqref{norm_Qk}.
Then, the desired result holds with $b = \zeta \norm{Q_0}$.
\end{proof}

Now, we show how we can lower bound the decrease in the regularized model and in the objective function.

\begin{lemma}\label{lemma:decr}
If Condition~\ref{cond:Q} is satisfied, then, for every iteration $k$ of Algorithm~\ref{alg:reg_method}, we have
\begin{enumerate}[label=(\roman*)]
\item $\displaystyle{m_k(0) - m_k(s_k) \ge c \norm{s_k}^2}$, \label{lemma:decr_m}
\item $\displaystyle{f(x_k) - f(x_{k+1}) \ge \eta c \norm{x_{k+1}-x_k}^2}$, \label{lemma:decr_f}
\end{enumerate}
where
\[
c := \Bigl(\frac12a-\tau\Bigr) > 0.
\]
\end{lemma}

\begin{proof}{Proof}
Fix any iteration $k \ge 0$.
Recalling the expression of $\nabla m_k$ given in~\eqref{grad_mk}, we can write
\begin{equation}\label{gm_s}
\norm{\nabla m_k(s_k)} \norm{s_k} \ge \nabla m_k(s_k)^T s_k = \nabla f(x_k)^T s_k + s_k^T Q_k s_k + \sigma_k \norm{s_k}^r.
\end{equation}
Using the approximate optimality condition on $s_k$ given~\eqref{sk_cond_gmk}, we also have
\begin{equation}\label{gm_ub}
\norm{\nabla m_k(s_k)} \le \tau \norm{s_k} \min\{\norm{s_k},1\} \le \tau \norm{s_k}.
\end{equation}
From~\eqref{gm_s}--\eqref{gm_ub}, it follows that
\begin{equation}\label{gs}
-\nabla f(x_k)^T s_k \ge s_k^T Q_k s_k + \sigma_k \norm{s_k}^r - \tau \norm{s_k}^2.
\end{equation}
Hence, recalling the expression of $m_k$ given in~\eqref{mk}, we get
\[
\begin{aligned}
m_k(0) - m_k(s_k) & = -\nabla f(x_k)^T s_k - \frac 12 s_k^T Q_k s_k - \frac{\sigma_k}r \norm{s_k}^r \\
                      & \ge \frac 12 s_k^T Q_k s_k - \tau \norm{s_k}^2 + \frac{\sigma_k(r-1)}r \norm{s_k}^r \\
                      & \ge \Bigl(\frac 12 a - \tau\Bigr) \norm{s_k}^2 + \frac{\sigma_k(r-1)}r \norm{s_k}^r,
\end{aligned}
\]
where we have used~\eqref{gs} in the first inequality and~\eqref{Q_pos_def_assump} from Condition~\ref{cond:Q} in the second inequality.
Since, from the instructions of Algorithm~\ref{alg:reg_method}, we have $\sigma_k \ge 0$ and $r \ge 3$, then item~\ref{lemma:decr_m} is proved.

Item~\ref{lemma:decr_f} follows from item~\ref{lemma:decr_m} and the objective decrease condition~\eqref{f_decr}.
Finally, note that $c>0$ follows from~\eqref{a_tau} in Condition~\ref{cond:Q}.
\end{proof}

An immediate consequence of the above lemma is the monotonicity of $\{f(x_k)\}$. 

\begin{lemma}\label{lemma:f_conv}
Let $\{x_k\}$ be the sequence generated by Algorithm~\ref{alg:reg_method}.
If Condition~\ref{cond:Q} is satisfied, then $\{f(x_k)\}$ is monotonically non-increasing.
\end{lemma}

\begin{proof}{Proof}
It follows from item~\ref{lemma:decr_f} of Lemma~\ref{lemma:decr} recalling, from the instructions of Algorithm~\ref{alg:reg_method}, that $\eta > 0$.
\end{proof}

When $f$ is bounded from below, we show in the following two lemmas that $\{\norm{s_k}^p\}_{k \in \S}$ is summable for any $p \ge 2$, and consequently,
thanks to the approximate optimality condition on $s_k$ given in~\eqref{sk_cond_gmk}, $\{\norm{\nabla m_k(s_k)}\}_{k \in \S}$ is also summable.

\begin{lemma}\label{lemma:sk}
Let $\{x_k\}$ be the sequence generated by Algorithm~\ref{alg:reg_method}.
If Condition~\ref{cond:Q} is satisfied and $f$ is bounded from below, then there exists $T := \max_{k \in \S} \norm{s_k} < \infty$ and
\[
\sum_{k \in \S} \norm{s_k}^p < \infty \quad \forall p \ge 2.
\]
\end{lemma}

\begin{proof}{Proof}
The result is straightforward if $\S$ is a finite set, so let us consider the case where $\S$ is an infinite set.
From item~\ref{lemma:decr_f} of Lemma~\ref{lemma:decr}, we can write
\[
f(x_0) - \inf_{x \in \Rn} f(x) \ge f(x_0) - f(x_h) = \sum_{k=0}^{h-1} (f(x_k)-f(x_{k+1})) \ge \eta c \sum_{k=0}^{h-1}\norm{x_{k+1}-x_k}^2 \quad \forall h \ge 0.
\]
Letting $h \to \infty$, we obtain
\[
f(x_0) - \inf_{x \in \Rn} f(x) \ge \eta c \sum_{k=0}^{\infty} \norm{x_{k+1}-x_k}^2.
\]
Since $\inf_{x \in \Rn} f(x)$ is finite by assumption and recalling, from the instructions of Algorithm~\ref{alg:reg_method} and Lemma~\ref{lemma:decr}, that $\eta c>0$, we get
\[
\sum_{k=0}^{\infty} \norm{x_{k+1}-x_k}^2 < \infty.
\]
Hence, from the instructions of Algorithm~\ref{alg:reg_method} and the definition of $\S$ given in~\eqref{S_def}, it follows that
\[
\sum_{k=0}^{\infty} \norm{x_{k+1}-x_k}^2 = \sum_{k \in \S} \norm{x_{k+1}-x_k}^2 = \sum_{k \in \S} \norm{s_k}^2 < \infty.
\]
This implies that $\{\norm{s_k}\}_{k \in \S} \to 0$ and we can define $T < \infty$ as in the assertion, leading to
\[
\sum_{k \in \S} \norm{s_k}^p = \sum_{k \in \S} \norm{s_k}^2 \norm{s_k}^{p-2} \le T^{p-2} \sum_{k \in \S} \norm{s_k}^2 < \infty \quad \forall p \ge 2.
\]
\end{proof}

\begin{lemma}\label{lemma:gmk}
Let $\{x_k\}$ be the sequence generated by Algorithm~\ref{alg:reg_method}.
If Condition~\ref{cond:Q} is satisfied and $f$ is bounded from below, then
\[
\sum_{k \in \S} \norm{\nabla m_k(s_k)} < \infty.
\]
\end{lemma}

\begin{proof}{Proof}
The result is straightforward if $\S$ is a finite set, so let us consider the case where $\S$ is an infinite set.
From Lemma~\ref{lemma:sk}, it follows that $\{\|s_k\|\}_{k \in \S} \to 0$.
Then, using the approximate optimality condition on $s_k$ given~\eqref{sk_cond_gmk}, an iteration $\bar k \in \S$ exists such that
\[
\norm{\nabla m_k(s_k)} \le \tau \norm{s_k}^2 \quad \forall k \ge \bar k, \, k \in \S.
\]
The desired result follows by invoking Lemma~\ref{lemma:sk} again and noting that, from the instructions of Algorithm~\ref{alg:reg_method}, we have $\tau < \infty$.
\end{proof}

Let us conclude this subsection of preliminary results by relating, for every iteration $k$, the optimality violation $\norm{\nabla f(x_k)}$ to $\norm{s_k}$.
\begin{lemma}\label{lemma:sk_gk}
Let $\{x_k\}$ be the sequence generated by Algorithm~\ref{alg:reg_method}. If Condition~\ref{cond:Q} is satisfied, then
\[
\norm{\nabla f(x_k)} \le (\tau + b + \sigmamax T^{r-2}) \norm{s_k} \quad \forall k \in \S,
\]
where $b$ and $T$ are defined in Lemma~\ref{lemma:Q_ub} and Lemma~\ref{lemma:sk}, respectively.
\end{lemma}

\begin{proof}{Proof}
Fix any iteration $k \in \S$. Using the approximate optimality condition on $s_k$ given~\eqref{sk_cond_gmk}
and recalling the expression of $\nabla m_k$ given in~\eqref{grad_mk}, we get
\[
\begin{aligned}
\norm{\nabla f(x_k)} & = \norm{\nabla m_k(s_k) - Q_k s_k - \sigma_k \norm{s_k}^{r-2} s_k} \\
                     & \le \norm{\nabla m_k(s_k)} + \norm{Q_k s_k} + \norm{\sigma_k \norm{s_k}^{r-2} s_k} \\
                     & \le \tau \norm{s_k} \min\{\norm{s_k},1\} + (\norm{Q_k} + \sigma_k \norm{s_k}^{r-2}) \norm{s_k} \\
                     & \le (\tau + \norm{Q_k} + \sigmamax \norm{s_k}^{r-2}) \norm{s_k},
\end{aligned}
\]
where, in the last inequality, we have used the fact that $\sigma_k \le \sigmamax$ from the instructions of Algorithm~\ref{alg:reg_method}.
The desired result is hence obtained by observing that $\norm{Q_k}$ can be upper bounded by $b$ in view of Lemma~\ref{lemma:Q_ub},
while $\norm{s_k}^{r-2}$ can be upper bounded by $T^{r-2}$in view of Lemma~\ref{lemma:sk} and the fact that $r \ge 3$ from the instructions of Algorithm~\ref{alg:reg_method}.
\end{proof}

\subsection{Main Result}
To establish the convergence of the whole sequence $\{x_k\}$ generated by Algorithm~\ref{alg:reg_method},
we need to assume pseudoconvexity of the objective function $f$ and the existence of a minimizer, as stated below.

\begin{assumption}\label{assump:f}
The objective function $f$ is pseudoconvex and has a global minimizer.
\end{assumption}

We recall that pseudoconvexity of a continuously differentiable function $f$ means that
\[
\nabla f(x)^T(y-x) \ge 0 \quad \Rightarrow \quad f(y) \ge f(x) \quad \forall x,y \in \Rn,
\]
or equivalently,
\begin{equation}\label{pseudoconv}
f(y) < f(x) \quad \Rightarrow \quad \nabla f(x)^T(y-x) < 0 \quad \forall x,y \in \Rn.
\end{equation}
We also recall that a point $x^*$ is a global minimizer of a pseudoconvex function $f$
if and only if $\nabla f(x^*)=0$.

We now show that, under Assumption~\ref{assump:f} and Condition~\ref{cond:Q}, the sequence $\{x_k\}$ generated by Algorithm~\ref{alg:reg_method} is variable metric quasi-Fej{\'e}r monotone according to Definition~\ref{def:var_quasi_fejer}.

\begin{proposition}\label{prop:seq_alg_variable_quasi_fejer}
Let $\{x_k\}$ be the sequence generated by Algorithm~\ref{alg:reg_method}.
If Assumption~\ref{assump:f} holds and Condition~\ref{cond:Q} is satisfied,
then $\{x_k\}$ is a variable metric quasi-Fej{\'e}r monotone sequence with respect to $\F$ relative to $\{Q_k\}$, according to Definition~\ref{def:var_quasi_fejer}, where
\begin{equation}\label{set_F}
\F := \Bigl\{x \in \Rn \colon f(x) \le \lim_{k \to \infty} f(x_k)\Bigr\}.
\end{equation}
\end{proposition}

\begin{proof}{Proof}
First note that $\F$ is well defined and non-empty.
This follows from Lemma~\ref{lemma:f_conv} and Assumption~\ref{assump:f}, which ensure that $\{f(x_k)\}$ converges to a value greater than or equal to the minimum value of $f$.

Now, consider the sequences of scalars $\{\psi_k\}$, $\{\epsilon_k\}$ and $\{\theta_k\}$ defined as follow:
\begin{subequations}\label{seqsk}
\begin{align}
\{\psi_k\} & \text{ from Condition~\ref{cond:Q}, i.e, satisfying~\eqref{eta_summable}}, \\
\epsilon_k & =
    \begin{cases}
    (1+\psi_k)\norm{s_k}_{Q_k}^2 & \quad \text{if } k \in \S, \\
    0 & \quad \text{otherwise},
    \end{cases} \\
\theta_k & =
    \begin{cases}
    2(1+\psi_k)(\norm{\nabla m_k(s_k)} + \sigma_k \norm{s_k}^{r-1}) & \quad \text{if } k \in \S, \\
    0 & \quad \text{otherwise}.
    \end{cases}
\end{align}
\end{subequations}
To obtain the desired result, we want to show that all the hypothesis of Lemma~\ref{lemma:eq_var_quasi_fejer} are satisfied, with $\F$ given in the assertion,
by using $\{Q_k\}$ satisfying Condition~\ref{cond:Q} and $\{\psi_k\}$, $\{\epsilon_k\}$, $\{\theta_k\}$ defined as in~\eqref{seqsk}.
Namely, we want to show that
\begin{enumerate}[label=(\roman*)]
\item $\{Q_k\}$ satisfies~\eqref{Q_var_quasi_fejer}--\eqref{Q_bounded};
\item $\{\psi_k\}$, $\{\epsilon_k\}$ and $\{\theta_k\}$ satisfy~\eqref{eta_summable}, \eqref{eps_summable} and~\eqref{theta_summable}, respectively;
\item inequality~\eqref{eq_var_quasi_fejer_ineq} holds for all $y \in \F$.
\end{enumerate}

Let us first observe that $\{Q_k\}$ satisfies~\eqref{Q_var_quasi_fejer}--\eqref{Q_bounded} from~\eqref{Q_pos_def_assump}, \eqref{a_tau} and Lemma~\ref{lemma:Q_ub}.
Moreover, recalling~\eqref{seqsk}, the following holds:
\begin{itemize}
\item $\{\psi_k\}$ satisfies~\eqref{eta_summable} by definition;
\item From Lemma~\ref{lemma:Q_ub}, we have
    \[
    \epsilon_k \le \Bigl(1 + \sup_{k \in \S} \psi_k \Bigr)\norm{Q_k} \norm{s_k}^2 \le \Bigl(1 + \sup_{k \in \S} \psi_k\Bigr) b \norm{s_k}^2 \quad \forall k \in \S.
    \]
    Using Lemma~\ref{lemma:sk} and the fact that $0 \le \sup_{k \in \S} \psi_k < \infty$ since $\{\psi_k\}$ satisfies~\eqref{eta_summable},
    it follows that $\{\epsilon_k\}$ satisfies~\eqref{eps_summable};
\item From the fact that $0 \le \sup_{k \in \S} \sigma_k  \le \sigmamax < \infty$
    by the instructions of Algorithm~\ref{alg:reg_method}, we have
    \[
    \theta_k \le 2\Bigl(1 + \sup_{k \in \S} \psi_k \Bigr)(\norm{\nabla m_k(s_k)} + \sigmamax \norm{s_k}^{r-1}) \quad \forall k \in \S.
    \]
    Using Lemma~\ref{lemma:sk} with $p = r-1 \ge 2$, Lemma~\ref{lemma:gmk} and the fact that $0 \le \sup_{k \in \S} \psi_k < \infty$ 
	since $\{\psi_k\}$ satisfies~\eqref{eta_summable}, it follows that $\{\theta_k\}$ satisfies~\eqref{theta_summable}.
\end{itemize}
What is left to show is that~\eqref{eq_var_quasi_fejer_ineq} holds for all $y \in \F$.
Consider any iteration $k \ge 0$ and fix any $y \in \F$. Using~\eqref{norm_Qk_succ}, we have that
\begin{equation}\label{contr_quasi_fejer}
\norm{x_{k+1}-y}^2_{Q_{k+1}} \le (1+\psi_k)\norm{x_{k+1}-y}^2_{Q_k}.
\end{equation}
We distinguish two cases depending whether $x_{k+1} = x_k$ or $x_{k+1} \ne x_k$.
\begin{enumerate}[label=(\roman*)]
\item Assume that $x_{k+1} = x_k$. Using~\eqref{contr_quasi_fejer} together with the fact that $\{\theta_k\} \subseteq [0,\infty)$ and $\{\epsilon_k\} \subseteq [0,\infty)$,
	we can write
    \[
    \norm{x_{k+1}-y}^2_{Q_{k+1}} \le (1+\psi_k)\norm{x_k-y}^2_{Q_k} \le (1+\psi_k) \norm{x_k-y}_{Q_k}^2 + \theta_k \norm{x_k-y} + \epsilon_k.
    \]
    Then, \eqref{eq_var_quasi_fejer_ineq} holds.
\item Assume that $x_{k+1} \ne x_k$. From the instructions of Algorithm~\ref{alg:reg_method} and the definition of $\S$ given in~\eqref{S_def}, we have that
	\begin{equation}\label{sk_nonzero2}
    s_k = x_{k+1} - x_k \ne 0
    \end{equation}
	and
	\begin{equation}\label{k_succ}
	k \in \S.
	\end{equation}   
    Hence, using item~\ref{lemma:decr_f} of Lemma~\ref{lemma:decr} and the fact that $\eta>0$ from the instructions of Algorithm~\ref{alg:reg_method}, we have that
    \begin{equation}\label{fk}
    f(x_k) > f(x_{k+1}) \ge f(y),
    \end{equation}
	where the last inequality follows from the fact that $y \in \F$.
    Now, from simple calculations and using~\eqref{sk_nonzero2}, we can write
	\begin{equation}\label{eq_quasi_fejer}
    \begin{aligned}
    \norm{x_{k+1}-y}_{Q_k}^2 & = \norm{x_k-y}_{Q_k}^2 + \norm{x_{k+1}-x_k}_{Q_k}^2 + 2(x_k-y)^T Q_k (x_{k+1}-x_k) \\
                               & = \norm{x_k-y}_{Q_k}^2 + \norm{s_k}_{Q_k}^2 + 2(x_k-y)^T Q_k s_k.
    \end{aligned}
	\end{equation}
    In order to upper bound $2(x_k-y)^T Q_k s_k$ in~\eqref{eq_quasi_fejer},
	note that, recalling the expression of $\nabla m_k$ given in~\eqref{grad_mk}, we have
    \[
    Q_k s_k = \nabla m_k(s_k) - \sigma_k \norm{s_k}^{r-2} s_k - \nabla f(x_k).
    \]
    Hence,
	\[
    \begin{aligned}
    (x_k-y)^T Q_k s_k & = (\nabla m_k(s_k) - \sigma_k \norm{s_k}^{r-2} s_k)^T (x_k-y) + \nabla f(x_k)^T (y-x_k) \\
                        & \le (\norm{\nabla m_k(s_k)} + \sigma_k \norm{s_k}^{r-1}) \norm{x_k-y} + \nabla f(x_k)^T (y-x_k) \\
                        & \le (\norm{\nabla m_k(s_k)} + \sigma_k \norm{s_k}^{r-1}) \norm{x_k-y},
    \end{aligned}
	\]
    where the last inequality, recalling~\eqref{fk}, follows from the pseudoconvexity condition~\eqref{pseudoconv} of Assumption~\ref{assump:f}.
    Combining the above chain of inequalities with~\eqref{eq_quasi_fejer}, we obtain
    \[
    \norm{x_{k+1}-y}_{Q_k}^2 \le \norm{x_k-y}_{Q_k}^2 + \norm{s_k}_{Q_k}^2 + 2(\norm{\nabla m_k(s_k)} + \sigma_k \norm{s_k}^{r-1}) \norm{x_k-y}.
    \]
    Multiplying all terms by $(1+\psi_k)$ and using~\eqref{contr_quasi_fejer}, we get
	\[
    \begin{aligned}
    & \norm{x_{k+1}-y}_{Q_{k+1}}^2 \le \\
    & (1+\psi_k) \norm{x_k-y}_{Q_k}^2 + (1+\psi_k)\norm{s_k}_{Q_k}^2 + 2(1+\psi_k)(\norm{\nabla m_k(s_k)} + \sigma_k \norm{s_k}^{r-1}) \norm{x_k-y}.
    \end{aligned}
	\]
    Namely, using~\eqref{seqsk} and recalling~\eqref{k_succ}, we have
    \[
    \norm{x_{k+1}-y}_{Q_{k+1}}^2 \le (1+\psi_k) \norm{x_k-y}_{Q_k}^2 + \theta_k \norm{x_k-y} + \epsilon_k,
    \]
    thus satisfying~\eqref{eq_var_quasi_fejer_ineq}.
\end{enumerate}
\end{proof}

Finally, by combining Proposition~\ref{prop:seq_alg_variable_quasi_fejer} and Theorem~\ref{th:variable_quasi_fejer},
we can easily prove the main result of the paper, that is, that the whole sequence $\{x_k\}$ generated by Algorithm~\ref{alg:reg_method} converges
under Assumption~\ref{assump:f} and Condition~\ref{cond:Q}.

\begin{theorem}\label{th:conv}
Let $\{x_k\}$ be the sequence generated by Algorithm~\ref{alg:reg_method}.
If Assumption~\ref{assump:f} holds and Condition~\ref{cond:Q} is satisfied, then $x^* \in \Rn$ exists such that
\[
\lim_{k \to \infty} x_k = x^*.
\]
\end{theorem}

\begin{proof}{Proof}
From Proposition~\ref{prop:seq_alg_variable_quasi_fejer}, we have that
$\{x_k\}$ is a variable metric quasi-Fej{\'e}r monotone sequence with respect to $\F$ relative to $\{Q_k\}$,
where $\F$ is defined as in~\eqref{set_F}.
From point~\ref{prop_variable_quasi_fejer_bounded} of Theorem~\ref{th:variable_quasi_fejer}, it follows that $\{x_k\}$ is bounded and hence
has a limit point $x^*$.
Then, using the monotonicity of $\{f(x_k)\}$ from Lemma~\ref{lemma:f_conv} and the continuity of $f$, we get that $\{f(x_k)\} \to f(x^*)$, that is, $x^* \in \F$.
The desired result finally follows from item~\ref{prop_variable_quasi_fejer_conv} of Theorem~\ref{th:variable_quasi_fejer}.
\end{proof}

\section{An Example of a Convergent Algorithm}\label{sec:alg_example}
In Theorem~\ref{th:conv}, we have shown that the whole sequence $\{x_k\}$ generated by Algorithm~\ref{alg:reg_method} converges to a point $x^* \in \Rn$,
provided that the objective function $f$ is pseudoconvex and the sequence of matrices $\{Q_k\}$ satisfies Condition~\ref{cond:Q}.
Note that Theorem~\ref{th:conv} by itself does not say anything about the nature of $x^*$, that is, whether or not $x^*$ is a minimizer of $f$.
This is due to the fact that, in Algorithm~\ref{alg:reg_method}, we have considered a general framework where some details were intentionally left unspecified.

To guarantee that the whole sequence $\{x_k\}$ converges to \textit{a minimizer} of a pseudoconvex function $f$,
we need to refine Algorithm~\ref{alg:reg_method} in order to guarantee that
\begin{equation}\label{g_to_zero}
\lim_{k \to \infty} \norm{\nabla f(x_k)} = 0.
\end{equation}
The above limit, together with Theorem~\ref{th:conv}, will ensure that the whole sequence $\{x_k\}$ converges to $x^*$ and the latter is a minimizer of $f$.
In the next proposition, we show that this surely occurs when we have an infinite number of successful iterations.

\begin{proposition}\label{prop:conv_opt}
Let $\{x_k\}$ be the sequence generated by Algorithm~\ref{alg:reg_method} and assume that $\S$ is an infinite set.
If Assumption~\ref{assump:f} holds and Condition~\ref{cond:Q} is satisfied, then
\[
\lim_{k \to \infty} x_k = x^*
\]
such that $x^*$ is a minimizer of $f$.
\end{proposition}

\begin{proof}{Proof}
Theorem~\ref{th:conv} ensures that $x^* \in \Rn$ exists such that $\{x_k\} \to x^*$.
The continuity of $\nabla f$ hence implies that $\{\nabla f(x_k)\} \to \nabla f(x^*)$.
Therefore, using the fact that $\S$ is an infinite set, from Lemma~\ref{lemma:sk_gk} we can write
\[
\norm{\nabla f(x^*)} = \lim_{k \to \infty} \norm{\nabla f(x_k)} =  \lim_{k \in \S} \norm{\nabla f(x_k)} \le (\tau + b + \sigmamax T^{r-2}) \lim_{k \in \S} \norm{s_k},
\]
with $\tau + b + \sigmamax T^{r-2} < \infty$.
Since Lemma~\ref{lemma:sk} implies that $\{\norm{s_k}\}_{k \in \S}$ converges to $0$, it follows that $\norm{\nabla f(x^*)} = 0$.
As $f$ is pseudoconvex from Assumption~\ref{assump:f}, we conclude that $x^*$ is a minimizer of $f$.
\end{proof}

In Algorithm~\ref{alg:reg_method_v2}, we provide an example of a scheme that uses all the features of Algorithm~\ref{alg:reg_method}
and guarantees the convergence to a minimizer of a pseudoconvex function $f$ by leveraging Proposition~\ref{prop:conv_opt}.

\begin{algorithm}
\caption{Example of a regularized method with convergence guarantees}\label{alg:reg_method_v2}
\begin{algorithmic}[1]
\State given $r \in [3,\infty)$, $\tau \in [0,\infty)$ and $\sigmamax \in [0,\infty)$
\State choose $x_0 \in \Rn$
\For{$k = 0,1,\ldots$}
\State choose any $\sigma_k \in [0,\sigmamax]$
\State compute a symmetric matrix $Q_k \in \R^{n \times n}$
\State compute $s_k \in \Rn$ satisfying~\eqref{sk_cond_gmk}
\State set $x_{k+1} = x_k + s_k$
\EndFor
\end{algorithmic}
\end{algorithm}

Note that, for every iteration $k$ of Algorithm~\ref{alg:reg_method_v2}, we can choose $\sigma_k$ to be any value in $[0,\sigmamax]$ and
set $x_{k+1} = x_k + s_k$ without checking the new objective value $f(x_{k+1})$.
As shown below, the objective decrease condition~\eqref{f_decr} is guaranteed by assuming Lipschitz continuity of $\nabla f$
and strengthening the lower bound requirement on the parameter $a$ appearing in Condition~\ref{cond:Q}.

\begin{assumption}\label{assump:g_lips}
The gradient $\nabla f$ of the objective function $f$ is Lipschitz continuous with constant $L>0$, that is,
\[
\norm{\nabla f(x) - \nabla f(y)} \le L \norm{x-y} \quad \forall x,y \in \Rn.
\]
\end{assumption}

\begin{condition}\label{cond:Q_g}
Under Assumption~\ref{assump:g_lips}, the same requirements expressed in Condition~\ref{cond:Q} hold, but~\eqref{a_tau} is replaced with
\begin{equation}\label{a_tau_L_eta}
a \ge 2\tau + \frac L{(1-\eta)},
\end{equation}
where $\eta \in (0,1)$ and $L$ is the Lipschitz constant appearing in Assumption~\ref{assump:g_lips}.
\end{condition}

Now, under Assumptions~\ref{assump:f}--\ref{assump:g_lips} and Condition~\ref{cond:Q_g},
we can show that the whole sequence $\{x_k\}$ generated by Algorithm~\ref{alg:reg_method_v2} converges to a minimizer of a pseudoconvex function $f$.

\begin{theorem}\label{th:conv_1st_order}
Let $\{x_k\}$ be the sequence generated by Algorithm~\ref{alg:reg_method_v2}.
If Assumptions~\ref{assump:f}--\ref{assump:g_lips} hold and Condition~\ref{cond:Q_g} is satisfied, then all iterations belong to $\S$ and
\[
\lim_{k \to \infty} x_k = x^*
\]
such that $x^*$ is a minimizer of $f$.
\end{theorem}

\begin{proof}{Proof}
Let us show that, under the assumptions and conditions stated, Algorithm~\ref{alg:reg_method_v2} is a special case of Algorithm~\ref{alg:reg_method}.
Namely, we want to show that the objective decrease condition~\eqref{f_decr} is satisfied for every iteration $k$ of Algorithm~\ref{alg:reg_method_v2}
with the same $\eta$ appearing in Condition~\ref{cond:Q_g}.

Proceeding by contradiction, assume that an iteration $k \ge 0$ exists such that~\eqref{f_decr} does not hold.
Reasoning as in the proof of item~\ref{lemma:decr_m} of Lemma~\ref{lemma:decr}, we still get
\begin{equation}\label{decr_m}
m_k(0) - m_k(s_k) \ge \Bigl(\frac12a-\tau\Bigr) \norm{s_k}^2.
\end{equation}
Then, using~\eqref{a_tau_L_eta} from Condition~\ref{cond:Q_g}, it follows that
$m_k(0) - m_k(s_k) > 0$ whenever $\norm{s_k} > 0$.
Hence, since $x_{k+1} = x_k + s_k$ from the instructions of Algorithm~\ref{alg:reg_method_v2}, the fact that~\eqref{f_decr} does not hold implies that
\begin{equation}\label{f_decr_contr}
f(x_k) - f(x_k+s_k) = f(x_k) - f(x_{k+1}) < \eta({m_k(0) - m_k(s_k)}).
\end{equation}
By the Lipschitz continuity of $\nabla f$ from Assumption~\ref{assump:g_lips}, using known results~\cite{nesterov:2013} we can write
\[
\begin{aligned}
f(x_k+s_k) & \le f(x_k) + \nabla f(x_k)^T s_k + \frac L2 \norm{s_k}^2 \\
					& \le f(x_k) + \nabla f(x_k)^T s_k + \frac L2 \norm{s_k}^2 + \frac 12 s_k^T Q_k s_k + \frac{\sigma_k}r \norm{s_k}^r,
\end{aligned}
\]
where the last inequality follows from the fact that $\frac 12 s_k^T Q_k s_k + \frac{\sigma_k}r \norm{s_k}^r \ge 0$
in view of the instructions of Algorithm~\ref{alg:reg_method_v2} together with Condition~\ref{cond:Q_g}.
Hence, recalling the expression of $m_k$, we get
\begin{equation}\label{f_decr_lips_g}
f(x_k+s_k) \le f(x_k) + \frac L2 \norm{s_k}^2 - ({m_k(0) - m_k(s_k)}).
\end{equation}
Using~\eqref{f_decr_contr} and~\eqref{f_decr_lips_g}, we get
\[
\frac L2 \norm{s_k}^2 \ge f(x_k+s_k) - f(x_k) + ({m_k(0) - m_k(s_k)}) > (1-\eta) ({m_k(0) - m_k(s_k)}).
\]
In view of~\eqref{decr_m}, it follows that
\[
\frac L2 \norm{s_k}^2 > (1-\eta) \Bigl(\frac12a-\tau\Bigr) \norm{s_k}^2,
\]
thus contradicting~\eqref{a_tau_L_eta} from Condition~\ref{cond:Q_g} and proving that Algorithm~\ref{alg:reg_method_v2} is a special case of Algorithm~\ref{alg:reg_method}.
Moreover, since $x_{k+1} = x_k + s_k$ for all $k \ge 0$ from the instructions of Algorithm~\ref{alg:reg_method_v2},
then all iterations are successful according to~\eqref{S_def}, that is,
\begin{equation}\label{U_empty}
\S = \{0,1,\ldots\}.
\end{equation}
Therefore, observing that Condition~\ref{cond:Q} is satisfied as it is implied by Condition~\ref{cond:Q_g},
it follows from Proposition~\ref{prop:conv_opt} that $\{x_k\} \to x^*$ such that $x^*$ is a minimizer of $f$.
\end{proof}

\subsection{Convergence Rate in the Convex Case}
Let us conclude the analysis of Algorithm~\ref{alg:reg_method_v2} by showing that, when $f$ is convex,
not only does the whole sequence of points $\{x_k\}$ converge to a minimizer, but the objective error also converges to zero at a rate of the order of $k^{-1}$.
Hence, we attain the same convergence rate for the objective error as other regularized methods that use first-order information in a convex setting~\cite{cartis:2012b}.

\begin{assumption}\label{assump:f_conv}
The objective function $f$ is convex and has a global minimizer.
\end{assumption}

\begin{proposition}\label{prop:rate_f_conv}
Let $\{x_k\}$ be the sequence generated by Algorithm~\ref{alg:reg_method_v2}.
If Assumptions~\ref{assump:g_lips}--\ref{assump:f_conv} hold and Condition~\ref{cond:Q_g} is satisfied, then
\[
f(x_k)-f^* \le \frac{R^2(f(x_0)-f^*)}{R^2 + \nu k(f(x_0)-f^*)} \quad \forall k \ge 1,
\]
where $f^* := \min_{x \in \Rn} f(x)$ and
\[
\nu := \frac{\eta c}{(\tau + b + \sigmamax T^{r-2})^2} > 0,
\]
while $R$, $b$, $c$ and $T$  are defined as in Theorem~\ref{th:variable_quasi_fejer}, Lemma~\ref{lemma:Q_ub}, Lemma~\ref{lemma:decr} and Lemma~\ref{lemma:sk},
respectively.
\end{proposition}

\begin{proof}{Proof}
Reasoning as in the proof of Theorem~\ref{th:conv_1st_order},
we have that Algorithm~\ref{alg:reg_method_v2} is a special case of Algorithm~\ref{alg:reg_method} where~\eqref{U_empty} holds.
Now fix an iteration $k \ge 0$. Using item~\ref{lemma:decr_f} of Lemma~\ref{lemma:decr} and the fact that $x_{k+1} = x_k + s_k$, we get
\[
f(x_k) - f(x_{k+1}) \ge \eta c \norm{s_k}^2.
\]
Since $k \in \S$ as~\eqref{U_empty} holds, using Lemma~\ref{lemma:sk_gk} we obtain
\begin{equation}\label{f_decr_g}
f(x_k) - f(x_{k+1}) \ge \nu \norm{\nabla f(x_k)}^2,
\end{equation}
where $\nu$ is defined as in the statement.

Moreover, using the convexity of $f$ from Assumption~\ref{assump:f_conv}, we can write
\begin{equation}\label{f_err_pseudoconv}
f(x_k)-f^* \le \nabla f(x_k)^T (x_k-x^*) \le \norm{\nabla f(x_k)} \norm{x_k-x^*}.
\end{equation}
Since Algorithm~\ref{alg:reg_method_v2} is a special case of Algorithm~\ref{alg:reg_method}, while 
Assumption~\ref{assump:f} and Condition~\ref{cond:Q} are implied by Assumptions~\ref{assump:f_conv} and Condition~\ref{cond:Q_g}, respectively,
then we can use Proposition~\ref{prop:seq_alg_variable_quasi_fejer} to get that 
$\{x_k\}$ is a variable metric quasi-Fej{\'e}r monotone sequence.
Using item~\ref{prop_variable_quasi_fejer_R0} of Theorem~\ref{th:variable_quasi_fejer}, it follows that $\norm{x_k-x^*} \le R$.
Hence, from~\eqref{f_err_pseudoconv} we obtain
\begin{equation}\label{f_err}
f(x_k)-f^* \le R \norm{\nabla f(x_k)}.
\end{equation}
Now, denote
\[
E_k := f(x_k) - f^*.
\]
It follows from~\eqref{f_decr_g} and~\eqref{f_err} that
\[
E_k - E_{k+1} \ge \frac{\nu}{R^2} E_k^2.
\]
Then,
\[
\frac 1{E_{k+1}} - \frac 1{E_k} = \frac{E_k - E_{k+1}}{E_{k+1}E_k} \ge \frac{\nu E_k}{R^2 E_{k+1}} \ge \frac{\nu}{R^2}.
\]
Applying the above chain of inequalities recursively, we get
\[
\frac 1{E_{k+1}} \ge \frac 1{E_0} + \frac{\nu (k+1)}{R^2} = \frac{R^2 + E_0 \nu (k+1)}{E_0R^2}.
\]
Then, the desired result follows.
\end{proof}

\begin{remark}
The $\mathcal O(1/k)$ convergence rate given in Proposition~\ref{prop:rate_f_conv} does not require bounded level sets, thus showing an improvement over classical results of regularized methods
that use first-order information in the convex case (cfr.~\cite[Theorem~2.5]{cartis:2012b}).
\end{remark}

\section{An iterative procedure for matrix computation}\label{sec:matrix}
The requirements given in Condition~\ref{cond:Q} on the sequence of matrices $\{Q_k\}$ are rather general and can be satisfied in many ways.
As highlighted in Remark~\ref{rem:Q}, the easiest choice is setting $Q_k$ to an appropriate constant matrix at any iteration $k$.

In this section, we describe a strategy based on an iterative update of $Q_k$.
In what follows, given a symmetric matrix $M \in \R^{n \times n}$, we denote its lowest eigenvalue by $\lambdamin(M)$.

Given $\{\psi_k\}$ and $a$ satisfying~\eqref{eta_summable} and~\eqref{a_tau}, respectively, let $\{a_k\}$ be a sequence of scalars such that
\begin{equation}\label{ak}
a < a_{k+1} < a_k \quad \forall k \ge 0.
\end{equation}
At iteration $k$, assume that we have a symmetric matrix $Q_k \in \R^{n \times n}$ such that
\begin{equation}\label{Qk_pd}
Q_k \succeq a_k I
\end{equation}
and we have computed a symmetric matrix $\tilde Q_{k+1} \in \R^{n \times n}$.
Now define
\[
\Delta_k := \tilde Q_{k+1} - Q_k
\]
and choose $\beta_k \ge 0$ to set
\begin{equation}\label{Qk_upd}
Q_{k+1} = Q_k + \beta_k \Delta_k = (1-\beta_k) Q_k + \beta_k \tilde Q_{k+1}.
\end{equation}
We will show that, choosing $\beta_k$ sufficiently small, the resulting sequence $\{Q_k\}$ satisfies Condition~\ref{cond:Q} provided that $Q_0 \succ a_0 I$.

In particular, since~\eqref{Q_pos_def_assump} holds by assumption in view of~\eqref{ak}--\eqref{Qk_pd}, what we need to show is that,
using suitable values of $\beta_k$, \eqref{Q_prec_assump} holds and~\eqref{Qk_pd} is satisfied with $k$ replaced by $k+1$ as well.
If this is true, reasoning by induction we get that Condition~\ref{cond:Q} is satisfied provided that $Q_0 \succ a_0 I$.
In the sequel, we will describe two possible strategies to choose $\beta_k$.

\subsection{First option}\label{subsec:matrix1}
At iteration $k$, assuming without loss of generality that $\norm{\Delta_k}>0$ (otherwise $Q_{k+1}=Q_k$), we can set
\begin{equation}\label{beta_k_ref}
\beta_k \in
\begin{cases}
\biggl[0, \dfrac{\psi_k \lambdamin(Q_k)}{\norm{\Delta_k}}\biggr] & \quad \text{if } \lambdamin(\tilde Q_{k+1}) \ge \lambdamin(Q_k), \\[6pt]
\biggl[0, \min\biggl\{\dfrac{\psi_k \lambdamin(Q_k)}{\norm{\Delta_k}},\dfrac{\lambdamin(Q_k)-a_{k+1}}{\lambdamin(Q_k)-\lambdamin(\tilde Q_{k+1})}\biggr\}\biggr]
& \quad \text{if } \lambdamin(\tilde Q_{k+1}) < \lambdamin(Q_k).
\end{cases}
\end{equation}
As discussed above, to satisfy Condition~\ref{cond:Q} provided that $Q_0 \succ a_0 I$, we need to prove that~\eqref{Q_prec_assump} holds and that~\eqref{Qk_pd} is satisfied with $k$ replaced by $k+1$ as well.
Indeed, at every iteration $k$ we have the following:
\begin{itemize}
\item Since $\beta_k \le \psi_k \lambdamin(Q_k)/\norm{\Delta_k}$ from~\eqref{beta_k_ref}, then
	\[
	\beta_k \Delta_k \preceq \psi_k Q_k.
	\]
	Adding $Q_k$ to both terms and using~\eqref{Qk_upd}, it follows that~\eqref{Q_prec_assump} holds.
\item Using~\eqref{Qk_upd} and the Weyl's inequality, we can write
	\begin{equation}\label{Qk_pd2}
	\lambdamin(Q_{k+1}) \ge (1-\beta_k) \lambdamin(Q_k) + \beta_k \lambdamin(\tilde Q_{k+1}) = \beta_k (\lambdamin(\tilde Q_{k+1})-\lambdamin(Q_k))+\lambdamin(Q_k).
	\end{equation}
	From~\eqref{Qk_pd2} and the non-negativity of $\beta_k$ expressed in~\eqref{beta_k_ref}, it follows that:
	\begin{itemize}
		\item if $\lambdamin(\tilde Q_{k+1}) \ge \lambdamin(Q_k)$, then $\lambdamin(Q_{k+1}) \ge \lambdamin(Q_k)$,
		implying from~\eqref{ak}--\eqref{Qk_pd} that
		\[
		\lambdamin(Q_{k+1}) \ge a_{k+1} \quad \forall \beta_k \ge 0;
		\]
		\item if $\lambdamin(\tilde Q_{k+1}) < \lambdamin(Q_k)$, then
		\[
		\lambdamin(Q_{k+1}) \ge a_{k+1} \quad \forall \beta_k \le \dfrac{\lambdamin(Q_k)-a_{k+1}}{\lambdamin(Q_k)-\lambdamin(\tilde Q_{k+1})}.
		\]
	\end{itemize}
	From the above two cases, we get that~\eqref{beta_k_ref} guarantees $\lambdamin(Q_{k+1}) \ge a_{k+1}$, that is, \eqref{Qk_pd} is satisfied with $k$ replaced by $k+1$ as well.
\end{itemize}
Note that, provided that $\psi_k>0$, the upper bound on $\beta_k$ obtained from~\eqref{beta_k_ref} is positive in view of~\eqref{ak}--\eqref{Qk_pd}.

Finally, also note that the acceptable values of $\beta_k$ obtained from~\eqref{beta_k_ref} may be $\ge 1$.
This can be useful, for example, if one intends to use a quasi-Newton formula to compute $\tilde Q_{k+1}$ since
we may set $Q_{k+1} = \tilde Q_{k+1}$ in such a case, hence recovering the standard quasi-Newton update.

\subsection{Second option}\label{subsec:matrix2}
The strategy described in the previous subsection requires to compute $\lambdamin(Q_k)$ and $\lambdamin(\tilde Q_{k+1})$ at iteration $k$.
In order to reduce the computational burden associated with the lowest eigenvalue computation, here we describe an alternative strategy which, at iteration $k$, requires to compute $\lambdamin(Q_k)$ only.
This comes at the price of requiring
\begin{equation}\label{tildeQk_spd}
\tilde Q_{k+1} \succeq 0
\end{equation}
and, as to be discussed below, restricting the acceptable values of $\beta_k$.

In more detail, assuming again without loss of generality that $\norm{\Delta_k}>0$ (otherwise $Q_{k+1}=Q_k$), we can choose
\begin{equation}\label{beta_k}
0 \le \beta_k \le \min\biggl\{\frac{\psi_k \lambdamin(Q_k)}{\norm{\Delta_k}},1-\frac{a_{k+1}}{\lambdamin(Q_k)}\biggr\}.
\end{equation}
As discussed above, to satisfy Condition~\ref{cond:Q} provided that $Q_0 \succ a_0 I$, we need to prove that~\eqref{Q_prec_assump} holds and that~\eqref{Qk_pd} is satisfied with $k$ replaced by $k+1$ as well.
Indeed, at every iteration $k$ we have the following:
\begin{itemize}
\item Reasoning as in Subsection~\ref{subsec:matrix1}, since $\beta_k \le \psi_k \lambdamin(Q_k)/\norm{\Delta_k}$ from~\eqref{beta_k}, then
	\[
	\beta_k \Delta_k \preceq \psi_k Q_k.
	\]
	Adding $Q_k$ to both terms and using~\eqref{Qk_upd}, it follows that~\eqref{Q_prec_assump} holds.
\item Since $\beta_k \le 1 - a_{k+1}/\lambdamin(Q_k)$ from~\eqref{beta_k}, and using the fact that $\lambdamin(Q_k)>0$ from~\eqref{ak}--\eqref{Qk_pd}, we have
	\begin{equation}\label{Qk_pd1}
	(1-\beta_k) \lambdamin(Q_k) \ge a_{k+1}.
	\end{equation}
	Moreover, reasoning as in Subsection~\ref{subsec:matrix1}, using~\eqref{Qk_upd} and the Weyl's inequality, we get~\eqref{Qk_pd2}.
	In particular, since $\beta_k \lambdamin(\tilde Q_{k+1}) \ge 0$ from~\eqref{tildeQk_spd} and the non-negativity of $\beta_k$ expressed in~\eqref{beta_k},
	from~\eqref{Qk_pd2} we obtain
	\begin{equation}\label{Qk_pd3}
	\lambdamin(Q_{k+1}) \ge (1-\beta_k) \lambdamin(Q_k).
	\end{equation}
	Combining~\eqref{Qk_pd1} and~\eqref{Qk_pd3}, we get $\lambdamin(Q_{k+1}) \ge a_{k+1}$, that is, \eqref{Qk_pd} is satisfied with $k$ replaced by $k+1$ as well.
\end{itemize}
Note that, from~\eqref{ak}--\eqref{Qk_pd}, we have
\[
0 < a_{k+1} < \lambdamin(Q_k).
\]
Then, provided that $\psi_k>0$, the upper bound on $\beta_k$ obtained from~\eqref{beta_k} is positive, as for the strategy outlined in Subsection~\ref{subsec:matrix1}.
However, here the acceptable values of $\beta_k$ obtained from~\eqref{beta_k} are $<1$.
Namely, using~\eqref{beta_k} we can never set $Q_{k+1} = \tilde Q_{k+1}$, differently from the strategy described in Subsection~\ref{subsec:matrix1}.

\section{Numerical examples}\label{sec:num}
In this section, we analyze the practical performance of the proposed algorithmic framework using different strategies
to compute a sequence of matrices $\{Q_k\}$ satisfying Condition~\ref{cond:Q}. The experiments were run in Matlab R2025b on an Apple MacBook Pro with an Apple M1 Pro Chip and 16 GB RAM.

We consider the ARC algorithm~\cite{cartis:2011a}, which is a special case of Algorithm~\ref{alg:reg_method} with $r=3$. In particular, at each iteration $k$:
\begin{itemize}
\item $s_k$ is computed to satisfy, together with condition~\eqref{sk_cond_gmk} required by our scheme, also the Cauchy condition
	\[
	m_k(s_k) \le m_k(s_k^\text{C}), \quad \text{where} \quad s_k^{\text{C}} = -\alpha_k^{\text{C}} \nabla f(x_k), \quad \alpha_k^{\text{C}} \in \argmin_{\alpha \ge 0} m_k(-\alpha \nabla f(x_k));
	\]
\item after computing $\rho_k$ as in~\eqref{rho1}, we set
	\[
	x_{k+1} =
	\begin{cases}
	x_k + s_k \quad & \text{if } \rho_k \ge \eta_1, \\
	x_k \quad & \text{otherwise},
	\end{cases}
	\]
	and
	\[
	\sigma_{k+1} =
	\begin{cases}
	\max\{\sigma_k/2,\sigmamin\} \quad & \text{if } \rho_k \ge \eta_2, \\
	\sigma_k \quad & \text{if } \eta_1 \le \rho_k < \eta_2, \\
	2\sigma_k \quad & \text{if } \rho_k < \eta_1,
	\end{cases}
	\]
	with $\eta_1 = 0.1$, $\eta_2 = 0.9$, $\sigmamin = 10^{-6}$ and $\sigma_0 = 1$.
\end{itemize}

We consider four versions of ARC, each using a different strategy to compute $Q_k$ at iteration $k$ among those analyzed in Sections~\ref{sec:alg_example}--\ref{sec:matrix}:
\begin{itemize}
\item ARC$_{1}$ uses the strategy described in Section~\ref{sec:alg_example} with a constant matrix $Q_k = aI$, where $a$ is such that~\eqref{a_tau_L_eta} holds with equality.
% Namely, ARC$_{1}$ is a special case of Algorithm~\ref{alg:reg_method_v2}.
\item ARC$_{2}$ uses the strategy described in Section~\ref{sec:matrix}, where $\tilde Q_{k+1}$ is the BFGS update~\cite{nocedal:2006} and $\beta_k$ is computed as described in Subsection~\ref{subsec:matrix1}.
In particular, $\beta_k$ is chosen as the minimum between $1$ and the largest value obtained from~\eqref{beta_k_ref} so as to obtain the standard BFGS update (i.e., $\beta_k=1$) whenever possible.
\item ARC$_{3}$ differs from ARC$_{2}$ in that $\beta_k$ is computed as described in Subsection~\ref{subsec:matrix2}. In particular, $\beta_k$ is chosen as the largest value obtained from~\eqref{beta_k} (which is surely $<1$).
\item ARC$_{4}$ differs from ARC$_{2}$ in that $\tilde Q_{k+1}$ is a diagonal matrix obtained from the diagonal part of the ordinary BFGS update~\cite{li:2022}, with $Q_k$ diagonal as well.
Note that,  in this case, computing $\lambdamin(Q_k)$ and $\lambdamin(\tilde Q_{k+1})$ has a negligible cost.
\end{itemize}

Observe that ARC$_{1}$ is a special case of Algorithm~\ref{alg:reg_method_v2} which trivially satisfies Condition~\ref{cond:Q_g} and, consequently, Condition~\ref{cond:Q},
regardless of how $\{\psi_k\}$ is chosen. Moreover, in view of Theorem~\ref{th:conv_1st_order}, all iterations of ARC$_{1}$ are successful and $\{\sigma_k\}$ is therefore bounded from above by $\sigma_0$.

For the other variants, which implement the strategies described in Section~\ref{sec:matrix}, we use $\{\psi_k\} = \{10^8(k+1)^{-2}\}$ and $\{a_k\} = \{a + (k+1)^{-1}\}$, with $a = 2\tau + 10^{-6}$.
Moreover, $\sigma_k$ remains below $\sigmamax = 10^{15}$ in all our experiments. However, in case $\sigma_k$ exceeds a pre-specified $\sigmamax$, one may switch to ARC$_{1}$.

We consider the following two classes of pseudoconvex problems that admit multiple optimal solutions:
\begin{itemize}
\item \textit{Binary classification with squared hinge loss.}
	Let $\{v_1,\ldots,v_m\} \subseteq \Rn$, with $m=20,000$ and $n = 1000$, be a set of vectors that can be divided into two equally sized groups by a hyperplane passing through the origin with margin $0.01$.
	Using labels $y^1,\ldots,y^m \in \{\pm 1\}$ assigned to the points, we want to compute a separating hyperplane by minimizing a squared hinge loss as in the following problem:
	\[
	\min f(x) = \sum_{i=1}^m \max\{0, 1-y^i (v^i)^T x\}^2.
	\]
	Note that $f$ is convex and any separating hyperplane passing through the origin, after an appropriate rescaling, defines an optimal solution.
\item \textit{Linear feasibility with log-penalty.} Let $A \in \R^{m \times n}$ be the matrix of a linear system, with $m = 100$ and $n = 1000$. We want to compute a solution of $Ax=0$ by minimizing a log-penalty loss as in the following problem:
	\[
	\min f(x) = \log(1+\|Ax\|^2).
	\]
	Note that $f$ is pseudoconvex and any vector in the kernel of $A$ is an optimal solution.
\end{itemize}

For both classes of problems, we run the considered algorithms on 10 randomly generated instances, and the average results are reported in Figure~\ref{fig:plot}. We observe that ARC$_1$, which uses a constant matrix at all iterations, performs the worst on both problems. Among the other three versions, ARC$_2$ and ARC$_3$ yield very similar results and outperform ARC$_4$, especially on the first class of problems.
Overall, these results seem to suggest that a BFGS-type approach can be beneficial for computing the sequence of matrices $\{Q_k\}$.

\begin{figure}
\centering
\subfloat[]
{\includegraphics{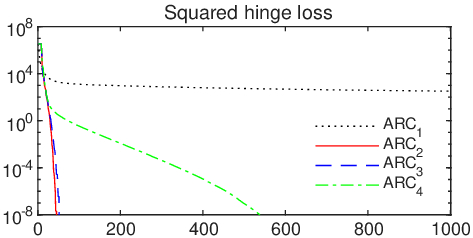}} \quad
\subfloat[]
{\includegraphics{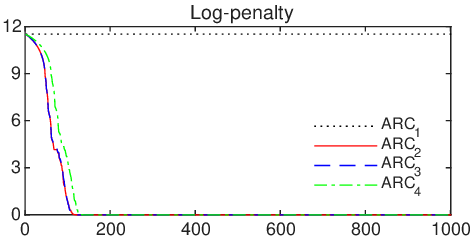}}
\caption{Objective error decrease on (a) binary classification problems using squared hinge loss and (b) linear feasibility problems with log-penalty. In the left plot, the $y$ axis is in logarithmic scale.}
\label{fig:plot}
\end{figure}

\section{Conclusions}\label{sec:conclusions}
In this paper, we have established some convergence properties of regularized methods for unconstrained optimization, that is,
algorithms that use quadratic models regularized by a power $r \ge 3$ of the norm of the step,
focusing on the case where only the objective function and its gradient are evaluated.
We have shown that, if the objective function is pseudoconvex and the regularized model is properly chosen at every iteration,
then the whole sequence of points generated by this class of algorithms converges.
Moreover, the convergence does not require bounded level sets.

The key to obtain the above result is choosing, at every iteration, the matrix of the quadratic model in such a way that the resulting sequence of points
is variable metric quasi-Fej{\'e}r monotone.
% The choice of such a matrix can be simple in practice, e.g., we can use a constant positive definite matrix for all iterations.

We finally observe that the convergence of the whole sequence of points was established in~\cite{cartis:2011a,nesterov:2006}
with $r = 3$ using second-order information.
Specifically, the result was obtained in~\cite[Theorem~4.4]{cartis:2011a} when the Hessian of $f$ is continuous and there exists
a subsequence of iterates converging to a point where the Hessian of $f$ is non-singular;
whereas the result was obtained in~\cite[Theorem~3]{nesterov:2006} by employing exact minimizers of the regularized model built with the Hessian matrix of $f$,
the latter assumed to be Lipschitz continuous and such that its smallest eigenvalue
at the starting point is sufficiently large compared to the gradient norm.
We see that our analysis differs mainly in that we do not require $f$ to have second-order derivatives (and we allow for a regularization power $r \ge 3$), although we impose the additional pseudoconvexity condition on $f$.
Moreover, compared to~\cite{cartis:2011a}, our results do not need to assume specific properties at the limit,
whereas, compared to~\cite{nesterov:2006}, our results are independent of the starting point and use approximate minimizers of the regularized model.

\bibliography{cristofari2026biblio}

\end{document}